\documentclass{SCAE}%SCAEOL for online version; SCAE for publication version; SCAES for the paper dedicated to somebody.
\numberwithin{equation}{section}
\usepackage{amsmath}
\usepackage{graphicx}
\usepackage{enumerate}
\usepackage{subfig}
\usepackage{multicol}

\newcommand{\R}{\mathbb{R}}
\newcommand{\h}{\mathcal{H}}
\newcommand{\M}{\mathcal{M}}
\newcommand{\tw}{\mathcal{T}_\omega}
\newcommand{\intrn}{\int_{\R^N}}
\newcommand{\abr}[1]{\langle #1\rangle}
\newcommand{\ms}{\mathcal{S}}

\newcommand{\ba}{\begin{aligned}}
\newcommand{\ea}{\end{aligned}}

\begin{document}

%Basic Information
\Year{2015} %
\Month{January}
\Vol{58} %
\No{1} %
\BeginPage{1} %
\EndPage{XX} %
\AuthorMark{TIAN R S {\it et al.}}
\ReceivedDay{November 17, 2014}
\AcceptedDay{April 21, 2015}
\PublishedOnlineDay{; published online January 22, 2014}
\DOI{10.1007/s11425-000-0000-0} % The author doesn't need fill in it.

% \title[short text for running head]{full title}{comments for title}
\title{Existence and bifurcation of solutions for a double coupled system of Schr\"odinger equations}{}

% \author[]{Full name}{footnote}
% Remark:  One \author for one author

\author{TIAN RuShun}{}
\author{ZHANG ZhiTao}{Corresponding author}

%
% \address[{\rm1}]{Academy of Mathematics and Systems Science, Chinese Academy of Sciences,
% Beijing {\rm 100190}, P. R. China;}
\address{Academy of Mathematics and Systems Science, and Hua
Loo-Keng Key Laboratory of Mathematics,\\ Chinese Academy of Sciences,
Beijing {\rm 100190}, P. R. China}

\Emails{rushun.tian@amss.ac.cn, zzt@math.ac.cn}
\maketitle

%     Abstract is required.

 {\begin{center}
\parbox{14.5cm}{\begin{abstract}
 Consider the following system of double coupled Schr\"odinger
equations arising from Bose-Einstein condensates etc.,
 \begin{equation*}
  \left\{\begin{array}{l}
           -\Delta u + u  =\mu_1 u^3 + \beta uv^2- \kappa v,\\
           -\Delta v + v  =\mu_2 v^3 + \beta u^2v- \kappa u,\\
           u\neq0, v\neq0\ \hbox{and}\ u, v\in H^1(\R^N),
          \end{array}
   \right.
 \end{equation*}where $\mu_1, \mu_2$ are positive and fixed; $\kappa$ and $\beta$ are linear and nonlinear coupling parameters respectively. We first use critical point theory and Liouville type theorem
 to prove some existence and nonexistence results on the positive solutions
 of this system. Then using the positive and non-degenerate solution of the scalar equation $-\Delta\omega+\omega=\omega^3$, $\omega\in H_r^1(\R^N)$, we construct a synchronized solution branch to prove
 that for $\beta$ in certain range and fixed, there exist a series of bifurcations in product space $\R\times H^1_r(\R^N)\times H^1_r(\R^N)$ with parameter $\kappa$.\vspace{-3mm}
\end{abstract}}\end{center}}

%  Keyword is required.

%  \subjclass is required.

%%%%%%%%%%%%%%%%%%%%%%%%%%%%%%%%%%%%%%%%%%%%%%%%%%%%%%%%%%%%
\renewcommand{\baselinestretch}{1.2}
\begin{center} \renewcommand{\arraystretch}{1.5}
{\begin{tabular}{lp{0.8\textwidth}} \hline \scriptsize
{\bf Citation:}\!\!\!\!&\scriptsize TIAN R S, ZHANG Z T. Science  China: Mathematics  title. Sci China Math, 2014, 57, doi: 10.1007/s11425-000-0000-0\vspace{1mm}
\\
\hline
\end{tabular}}\end{center}

%%%%%%%%%%%%%%%%%%%%%%%%%%%%%%%%%%%%%%%%%%%%%%%%%%%%%%%%%%%%
%% Text of article.
%%%%%%%%%%%%%%%%%%%%%%%%%%%%%%%%%%%%%%%%%%%%%%%%%%%%%%%%%%%%
%    Section headings
\baselineskip 11pt\parindent=10.8pt  \wuhao
\section{Introduction}
In this paper, we study the following elliptic system
 \begin{equation}\label{equ:twoparameters}
  \left\{\begin{array}{l}
           -\Delta u + u  = \mu_1u^3 + \beta uv^2- \kappa v,\\
           -\Delta v + v= \mu_2v^3 + \beta u^2v -\kappa u ,\\
           u\neq0, v\neq0\ \hbox{and}\ u, v\in H^1(\R^N),
          \end{array}
   \right.
 \end{equation} where $\kappa, \beta$ are coupling parameters, $\mu_2\geq\mu_1>0$ are constants, $N=2,3$. System \eqref{equ:twoparameters} describes the standing wave solutions of coupled
 Schr\"odinger equations, which have wide applications in physics, such as nonlinear optics and Bose-Einstein condensates. See
 \cite{Esry-Greene-Burke-Bohn:1997, Mitchell-Chen-Shih-Segev:1996, Ruegg-Cavadini-etc:2003} and references therein for more details.

In the presence of only one coupling parameter, either $\kappa$ or $\beta$, extensive research has been done regarding the existence, multiplicity and asymptotic behavior of nontrivial solutions to
 \eqref{equ:twoparameters}. The obtained results are very interesting and important, we refer to \cite{Ambrosetti-Cerami-Ruiz:2008}-\cite{Ambrosetti-Colorado:2007},
 \cite{Bartsch-Wang:2006}-\cite{Dancer-Wei-Weth:2010}, \cite{Lin-Wei-CMP:2005}-\cite{Maia-Montefusco-Pellacci:2006},\cite{NTTV1}-\cite{NTTV2},\cite{Sirakov:2007}-\cite{Wei-Weth:2007}
 and references therein. If $\kappa\beta\neq0$, i.e., the linearly coupling terms and nonlinearly coupling terms both exist, to our best knowledge, analogous research is almost empty. The current paper is devoted to this double coupled case. The main goal of this paper is two-fold:
\begin{enumerate}[(i)]
 \item determine regions in $\kappa\beta$-plane for existence and nonexistence of positive solutions to \eqref{equ:twoparameters};
 \item refine the existence results in the case $\mu_1=\mu_2$ by using bifurcation theory, and get more quantitative descriptions for these solutions.
\end{enumerate}

We begin with some notations. Denote by $|\cdot|_p$ the usual  norm
of the space $L^p(\R^N)$ and $\h=H^1(\R^N)\times H^1(\R^N)$, where
$$H^1(\R^N)=\{u\in L^2(\R^N)| \nabla u\in L^2(\R^N)\}$$ is the
Hilbert space equipped with the following inner product and induced
norm $$\langle u, v\rangle:=\intrn(\nabla u\nabla
v+uv),\hspace{1cm}\|u\|:=\sqrt{(\intrn|\nabla u|^2+|u|^2)}.$$
Accordingly, the inner product and induced norm on $\h$ are given by
\begin{equation*}
 \langle(u,v),(\xi,\eta)\rangle=\intrn(\nabla u\nabla\xi+\nabla v\nabla\eta+u\xi+v\eta),\hspace{0.6cm}\|(u,v)\|_{\h}=\sqrt{\|u\|^2+\|v\|^2},
\end{equation*}
respectively. Let $\h_r:=H_r^1(\R^N)\times H_r^1(\R^N)$, where
$$H_r^1(\R^N)=\{u\in H^1(\R^N)~|~u~\hbox{is radially symmetric}\}.$$
A solution $(u, v)$ of \eqref{equ:twoparameters} is called a
positive solution if $u>0, v>0$ in $\R^N$. A solution of
\eqref{equ:twoparameters} is called a ground state solution if it
minimizes the energy functional $I_{\kappa, \beta}:\h\rightarrow\R$,
\begin{equation}\label{equ:energyfunctional}
 I_{\kappa, \beta}(u,v)=\frac{1}{2}(\|u\|^2+\|v\|^2)+\kappa\intrn uv-\frac{1}{4}\int_{\R^N}(\mu_1u^4+\mu_2v^4)- \frac{\beta}{2}\int_{\R^N}u^2v^2
\end{equation}
on the associated Nehari manifold
\begin{align*}
  \M &= \{(u, v)\in \h~|~\abr{I_{\kappa,\beta}'(u,v),(u,v)}=0,u\neq0,v\neq0\}\\
  &=\left\{(u,v)\in \h~\bigg|~ \|u\|^2+\|v\|^2+2\kappa\intrn uv=\intrn(\mu_1u^4+\mu_2v^4+2\beta u^2v^2),u\neq0,v\neq0\right\}.
\end{align*}
An important observation will be used later is that the ground state solution has positive energy if $-1<\kappa<0$. See Lemma \ref{thm:existence-asym}.

Our first result concerns the existence and nonexistence of positive solutions to \eqref{equ:twoparameters}. \iffalse Precisely, we find some regions in $\kappa\beta$-plane in which the existence and nonexistence of positive solutions can be determined.\fi

\begin{theorem}\label{thm:existence}
 System \eqref{equ:twoparameters} has
 \begin{enumerate}[(i)]
  \item no positive solution, if $\kappa<-1$ and $\beta\geq\bar{\beta}:=-(\mu_1^2\mu_2)^{1/3}$ or $\kappa=-1$ and $\beta>0$;
  \item one positive ground state solution if $\kappa\in(-1,0)$ and $\beta>0$.
 \end{enumerate}
 In particular, if $\mu_1=\mu_2$ (without loss of generality, assume $\mu_1=\mu_2=1$), \eqref{equ:twoparameters} has at least one positive solution in
  $\{(\kappa, \beta)~|~-1<\kappa\leq0, \beta\in\R\}\cup\{(\kappa, \beta)~|~\kappa>0, \beta>-1\}$.
\end{theorem}

% \medskip
The existence regions and nonexistence regions for positive solutions of \eqref{equ:twoparameters} in $\beta d$-plane are illustrated by Figure
\ref{pic:existence_kappa_beta} and Figure \ref{pic:existence_kappa_beta_radial} for asymmetric case and symmetric case respectively.

\begin{figure}[!ht]
 \centering
 \includegraphics[width=0.55\textwidth]{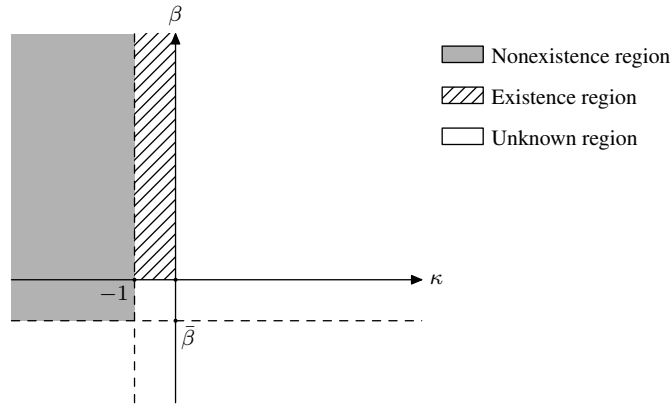}
 \caption{\footnotesize Existence and nonexistence of positive solutions when $\mu_1\neq\mu_2$}
 \label{pic:existence_kappa_beta}
\end{figure}
\begin{figure}[!ht]
 \centering
  \includegraphics[width=0.55\textwidth]{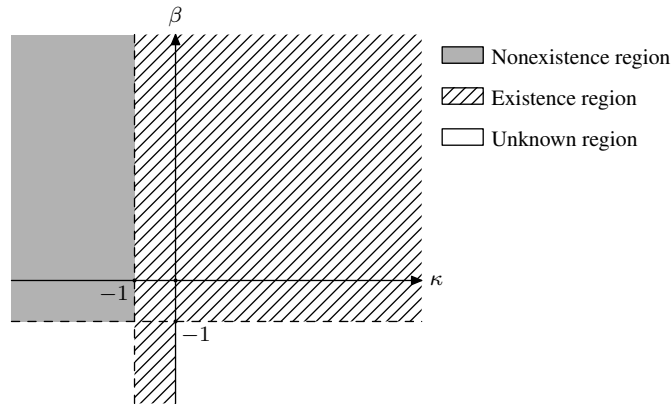}
 \caption{\footnotesize Existence and nonexistence of positive solutions when $\mu_1=\mu_2=1$,($\bar{\beta}=-1$)}
 \label{pic:existence_kappa_beta_radial}
\end{figure}
\begin{remark}
 In the case $\kappa=0$ and $\mu_1<\mu_2$, Bartsch and Wang \cite{Bartsch-Wang:2006} proved that \eqref{equ:twoparameters} has no positive solution if $\beta\in[\mu_1, \mu_2]$.
 We can see from Theorem \ref{thm:existence} (ii) that this nonexistence interval of $\beta$ vanishes when $\kappa$ becomes even sightly less than 0. Also in the case $\kappa=0$, if $\beta\notin[\mu_1, \mu_2]$, it is shown in \cite{Bartsch-Dancer-Wang:2010} that \eqref{equ:twoparameters} has at at least one positive solution in $\h_r$. This result is extended by Theorem \ref{thm:existence} (ii) to $-1<\kappa<0$ with the symmetric assumption $\mu_1=\mu_2$.
\end{remark}\par
We call a solution $(u, v)$ an {opposite sign solution} if $u>0,
v<0$ or $u<0, v>0$ in $\R^N$. It can be easily seen that system
\eqref{equ:twoparameters} is invariant under the following
transformation,
\begin{equation}\label{equ:action}
 \sigma: \R\times\R\times\h\rightarrow\R\times\R\times\h,\ \ \ \
 \sigma(\kappa, \beta, u, v)=(-\kappa, \beta, u, -v).
\end{equation}
 Using this $\sigma$-invariance of system \eqref{equ:twoparameters}, we immediately obtain a corollary of Theorem \ref{thm:existence} as follows.

\begin{corollary}
 System \eqref{equ:twoparameters} has
 \begin{enumerate}[(i)]
  \item no opposite sign solution, if $\kappa>1$ and $\beta\geq\bar{\beta}:=-(\mu_1^2\mu_2)^{1/3}$ or $\kappa=1$ and $\beta>0$;
  \item one opposite sign ground state solution if $\kappa\in(0, 1)$ and $\beta>0$.
 \end{enumerate}
 In particular, if $\mu_1=\mu_2=1$, \eqref{equ:twoparameters} has at least one pair of opposite sign solutions $(u,v)$ and $(-u,-v)$ in $\{(\kappa, \beta)~|~0\leq\kappa<1, \beta\in\R\}\cup\{(\kappa, \beta)~|~\kappa<0, \beta>-1\}$.
\end{corollary}

The symmetric assumption $\mu_1=\mu_2$ does not only provide better
existence results, but gives rise to two synchronized solution
branches as well. More precisely, we assume without loss of
generality $\mu_1=\mu_2=1$, system \eqref{equ:twoparameters} becomes
 \begin{equation}\label{equ:symmetric}
  \left\{\begin{array}{l}
           -\Delta u + u + \kappa v = u^3 + \beta uv^2,\\
           -\Delta v + v + \kappa u = v^3 + \beta u^2v,\\
           u\neq0, v\neq0\ \hbox{and}\ u, v\in H^1(\R^N).
          \end{array}
   \right.
 \end{equation}
Let $\omega$ be the non-degenerate positive solution of the scalar equation
\begin{equation}\label{equ:scalar_omega}
 -\Delta\omega+\omega=\omega^3,\hspace{1cm}\omega>0,\hspace{0.5cm} \omega\in H_r^1(\R^N).
\end{equation}
Then system \eqref{equ:symmetric} has solutions with two linearly dependent components, $u(x)=a_1\omega(bx)$ and $v(x)=a_2\omega(bx)$. Substituting this special solution in \eqref{equ:symmetric} and using the nondegeneracy of $\omega$, we get an algebraic system,
   \begin{equation*}
    \left\{\begin{array}{l}
            a_1+\kappa a_2=a_1b^2,\\
            a_2+\kappa a_1=a_2b^2,\\
            a_1^2+\beta a_2^2 = b^2,\\
            a_2^2+\beta a_1^2 = b^2.
           \end{array}
    \right.
   \end{equation*}
   By solving this system, we obtain two synchronized solution branches in $\R\times\R\times \h_r$,
   \begin{equation*}
    \tw^+=\left\{\left(\kappa, \beta, \sqrt{\frac{1+\kappa}{1+\beta}}\omega(\sqrt{1+\kappa}x), \sqrt{\frac{1+\kappa}{1+\beta}}\omega(\sqrt{1+\kappa}x)\right)\bigg|\beta>-1, \kappa>-1\right\},
   \end{equation*}
   \begin{equation*}
    \tw^-=\left\{\left(\kappa, \beta, \sqrt{\frac{1-\kappa}{1+\beta}}\omega(\sqrt{1-\kappa}x), -\sqrt{\frac{1-\kappa}{1+\beta}}\omega(\sqrt{1-\kappa}x)\right)\bigg|\beta>-1, \kappa<1\right\}.
   \end{equation*}
Clearly, $\tw^+$ is a positive solution branch and $\tw^-$ is an opposite sign solution branch. We will be focusing on bifurcations with respect to $\tw^+$, then the bifurcation results with respect to $\tw^-$ will follow immediately by taking advantage of the $\sigma$-invariance of \eqref{equ:symmetric}.

We use $\tw^+|_\kappa$ (respectively $\tw^+|_\beta$) to denote solution branch with fixed $\kappa$ (respectively with fixed $\beta$) and parameterized in $\beta$ (respectively parameterized in $\kappa$).
Clearly, $\tw^+|_\kappa, \tw^+|_\beta\subset\R\times\h_r$.

We call $(\beta_0, u_{\beta_0}, v_{\beta_0})\in\tw^+|_\kappa$ a
bifurcation point if there exist a sequence $(\beta_j, u_j,
v_j)\in\R\times\h_r\backslash\tw^+|_\kappa$ of solutions of
\eqref{equ:symmetric} such that
$(\beta_j,u_j,v_j)\rightarrow(\beta_0, u_{\beta_0}, v_{\beta_0})$ as
$j\rightarrow\infty$. The parameter value $\beta_0$ is called a
parameter bifurcation  point. Similarly, we call $(\kappa_0,
u_{\kappa_0}, v_{\kappa_0})\in\tw^+|_\kappa$ a bifurcation point if
there exist a sequence $(\kappa_j, u_j,
v_j)\in\R\times\h_r\backslash\tw^+|_\beta$ of solutions of
\eqref{equ:symmetric} such that
$(\kappa_j,u_j,v_j)\rightarrow(\kappa_0, u_{\kappa_0},
v_{\kappa_0})$ as $j\rightarrow\infty$. The parameter value
$\kappa_0$ is also called a parameter bifurcation point.

Denote the set of all nontrivial solutions by
$$\ms:=\{(\kappa,\beta, u, v)\in\R\times\R\times\h_r\backslash\tw^+ |\ (\kappa,\beta, u, v)\ \hbox{solves \eqref{equ:symmetric}}\}.$$ Also denote the restriction of $\ms$ with
fixed $\kappa$ by $\ms^\kappa$, and with fixed $\beta$ by $\ms^\beta$
respectively. We call $\beta_0$ (respectively $\kappa_0$) a global
bifurcation parameter if
\begin{enumerate}[(i)]
 \item there is a connected set of solutions of \eqref{equ:symmetric} bifurcates from $\tw^+|_\kappa$ at $(\beta_0, u_{\beta_0}, v_{\beta_0})$ (respectively from $\tw^+|_\beta$ at $(\kappa_0, u_{\kappa_0}, v_{\kappa_0})$) that is either unbounded in $\ms^\kappa$ (respectively $\ms^\beta$), or
 \item intersect with $\tw^+|_\kappa$ at a point other than $(\beta_0, u_{\beta_0}, v_{\beta_0})$ (respectively with $\tw^+|_\beta$ at a point other than $(\kappa_0, u_{\kappa_0}, v_{\kappa_0}))$.
\end{enumerate}
 These two cases are well known as Rabinowitz's global bifurcation alternatives.
% then $\beta_0$ is a global bifurcation parameter if the connected component $S|_\kappa(\beta)$ in $S|$

When $\kappa=0$, Bartsch, Dancer and Wang
\cite{Bartsch-Dancer-Wang:2010} proved the existence of infinitely
many bifurcation points with respect to $\tw^+|_{\kappa=0}$,  and
also gave descriptions for global bifurcation branches in radial
spaces. Similar results were also established in
\cite{Tian-Wang:2013-jan, Tian-Wang:2013-feb} for indefinite systems
and $\mu_1, \mu_2\in\R$. Recently, Bartsch \cite{Bartsch:2013}
considered a system with multiple components and established
corresponding bifurcation results. Note that, without the linearly
coupling terms, the existence of a synchronized solution branch does
not require symmetric assumption $\mu_1=\mu_2$. On the other hand,
when $\beta=0$ and $\kappa\neq0$, system \eqref{equ:twoparameters}
is linearly coupled. Ambrosetti, Ruiz and Cerami
\cite{Ambrosetti-Cerami-Ruiz:2008} gave descriptions on the ground
state solution of this type system in the cases $\kappa$ close to
$0$ and $-1$. In an unpublished manuscript by E. Abreu and Z.-Q. Wang,
the local bifurcation with respect to $\tw^+|_{\beta=0}$ and at
certain $\kappa\in(-1, 0)$ was studied.

\smallskip

In the current paper, we shall study the bifurcation phenomena of system \eqref{equ:symmetric} in the case $\kappa\neq0$ and $\beta\neq0$.

\begin{theorem}\label{thm:bifurcation}
 For any fixed $\beta\in(-1,0]$, system \eqref{equ:symmetric} has finitely many bifurcation points with respect to $\tw^+|_\beta$, where the bifurcation parameter $\kappa\in(-1, \infty)$. Moreover,
 \begin{enumerate}[(a)]
  \item the number of bifurcations and the $\h$-norm of bifurcating solutions on $\tw^+|_\beta$ both approach infinity as $\beta\rightarrow-1^+$;
  \item for each bifurcation point $(\kappa_l, u_{\kappa_l}, v_{\kappa_l})\in\tw^+|_\beta$, $l\geq1$, there is a global bifurcation branch $\ms^\beta_l\subset\R\times\h_r$.
%   $$ \{(\kappa_l, u_{\kappa_l}, v_{\kappa_l})\}=\overline{\ms_l}\cap\tw^+|_\beta.$$
  If $(\kappa, u, v)\in\ms^\beta_l$ with $-1<\kappa\leq0$, then $u>0$ and $v>0$.
 \end{enumerate}
\end{theorem}

\begin{remark}
 The positivity of bifurcation solutions is first proved for a modified system \eqref{equ:symmetric-cutoff}, which has the same bifurcations as system
 \eqref{equ:symmetric}.
\end{remark}

Using Theorem \ref{thm:bifurcation} and the $\sigma$-invariance of system \eqref{equ:symmetric}, we obtain the following corollary.

\begin{corollary}
For any fixed $\beta\in(-1,0]$, system \eqref{equ:symmetric} has finitely many local bifurcations with respect to $\tw^-|_\beta$,
where the bifurcation parameter $\kappa\in(-\infty, 1)$. Moreover,
 \begin{enumerate}[(a)]
  \item the number of bifurcations and the $\h$-norm of bifurcating solutions on $\tw^-|_\beta$ both approach infinity as $\beta\rightarrow1^-$;
  \item for each bifurcation point $(\kappa_l, u_{\kappa_l}, v_{\kappa_l})\in\tw^-|_\beta$, $l\geq1$, there is a global bifurcation branch $\ms_l^\beta\subset\R\times\h_r$.
%   $$  \{(\kappa_l, u_{\kappa_l}, v_{\kappa_l})\}=\overline{\ms_l}\cap\tw^-|_\beta.$$
  If $(\kappa, u, v)\in\ms^\beta_l$ with $0\leq\kappa<1$, then either $u>0$ and $v<0$, or $u<0$ and $v>0$.
 \end{enumerate}
\end{corollary}

This paper is organized as follows. In Section 2, we study the
existence and nonexistence of positive solutions of system
\eqref{equ:twoparameters}. Theorem \ref{thm:existence} will be
established through a few lemmas. \iffalse As an important step, two
synchronized solution branches, $\tw^+$ and $\tw^-$, shall be
constructed for proving the existence of positive solutions in the
symmetric case. \fi In Section 3, we find local bifurcations with
respect to $\tw^+|_\beta$. Moreover, we also give the number of
bifurcations and some estimates for the $\h$ norm of bifurcation
solutions as
 $\beta\rightarrow-1^+$.
In Section 4, we show the positivity of bifurcation solutions in a
certain range for the coupling parameters. Finally, Theorem
\ref{thm:bifurcation} is proved by combing the results of Section 3
and Section 4.

\section{Existence and nonexistence of positive solutions}
In this section, we study the existence and nonexistence of positive solutions to \eqref{equ:twoparameters} and \eqref{equ:symmetric} in terms of $\kappa$ and $\beta$. Lemma \ref{thm:nonexistence} and Lemma \ref{thm:existence-asym} hold for general system \eqref{equ:twoparameters}, therefore they also hold in the special case \eqref{equ:symmetric}. Lemma \ref{thm:existence-symm} is only proved for the symmetric system \eqref{equ:symmetric}.

\begin{lemma}\label{thm:nonexistence}
 System \eqref{equ:twoparameters} has no positive solution, if
 \begin{center}
  $\kappa<-1$ and $\beta\geq-(\mu_1^2\mu_2)^{1/3}$, or $\kappa=-1$ and $\beta>0$.
 \end{center}
\end{lemma}

\begin{proof} Assume for contradiction that $(u, v)$ is a positive solution of \eqref{equ:twoparameters}. {By the elliptic regularity theory and Sobolev embeddings, $u, v\in C^2(\R^N)$.}

Let us consider the case $\kappa\leq-1$, $\beta>0$ first. Add the two equations of \eqref{equ:twoparameters} together and set $U=u+v$, then we get
\begin{equation*}
 -\Delta U\geq-\Delta U +(1+\kappa) U= \mu_1u^3+\mu_2v^3+\beta u^2v+\beta uv^2\geq C U^3,
\end{equation*}
where $C=\min\{\mu_1, \beta/3\}$. Because $U\geq0$, according to
Liouville-type Theorem 2.3 of \cite{Dancer-Wei-Weth:2010}, we must
have $U\equiv0$. This is a contradiction.

Next, consider the case $\kappa<-1$ and $-(\mu_1^2\mu_2)^{1/3}\leq\beta\leq0$.
\iffalse It is easy to see that \eqref{equ:twoparameters} has no positive solution in the form $(u, (\mu_1/\mu_2)^{1/3}u)$. Otherwise, we see from the second equation,
\begin{equation*}
 -\Delta v\geq-\Delta v + \left[1+\kappa\left(\frac{\mu_2}{\mu_1}\right)^\frac{1}{3}\right]v =\left[\mu_2+\beta\left(\frac{\mu_2}{\mu_1}\right)^\frac{2}{3}\right]v^3\geq 0,
\end{equation*}
which indicates $v\equiv0$ (\tb{In the case $-\Delta v=0$, $v$ is bounded?}). \tr{Drop the part of $v$ and prove for $U$ directly? $-\Delta U +(1+\kappa) U\geq0$}\fi
Again, adding the two equations of \eqref{equ:twoparameters} together and setting $U=u+v$, we get
\begin{equation*}
 -\Delta U +(1+\kappa) U= \mu_1u^3+\mu_2v^3+\beta u^2v+\beta uv^2\geq \mu_1^\frac{1}{3}(\mu_1^\frac{1}{3}u-\mu_2^\frac{1}{3}v)^2U.
\end{equation*}
In the case $\kappa<-1$, Liouville-type Theorem 2.3 of
\cite{Dancer-Wei-Weth:2010} implies $U\equiv0$. \iffalse If
$\kappa=-1$, then $-\Delta U\geq0$ in $\R^N$. \tb{In this case, we
also have $U\equiv0$, due to the facts $\|U\|_{L^\infty}<\infty$ and
$U(x)\rightarrow0$ as $|x|\rightarrow\infty$. }\fi This is a
contradiction since $U\geq\max\{ u, v\}$ and $u, v$ are both
positive solutions.

The proof is completed. 
\end{proof}

\bigskip

The proof of the following existence lemma is modified from \cite{Ambrosetti-Cerami-Ruiz:2008}.

\begin{lemma}\label{thm:existence-asym}
 System \eqref{equ:twoparameters} has a positive ground state solution for $-1<\kappa<0$ and $\beta>0$.
\end{lemma}

\begin{proof}  The proof is divided into four steps.

Step 1: $I_{\kappa,\beta}|_{\M}$ has positive lower bound.

{\allowdisplaybreaks
Actually, by Sobolev embeddings and Cauchy inequality, \iffalse $H^1(\R^N)\hookrightarrow L^2(\R^N)$ and $H^1(\R^N)\hookrightarrow L^4(\R^N)$, \fi
\begin{align*}
 \|u\|^2+\|v\|^2+2\kappa\intrn uv&\geq\|u\|^2+\|v\|^2+2\kappa\intrn|u||v|\\
 &\geq(1+\kappa)(\|u\|^2+\|v\|^2)\\
 &\geq\frac{1+\kappa}{2}(\|u\|+\|v\|)^2,\\
 \intrn(\mu_1u^4+\mu_2v^4+2\beta u^2v^2)&\leq\intrn(\mu_1u^4+\mu_2v^4)+\beta\bigg(\intrn u^4+\intrn v^4\bigg)\\
 &\leq(\mu_2+\beta)(|u|^4_4+|v|^4_4)\\
 &\leq C(\mu_2+\beta)(\|u\|+\|v\|)^4,
\end{align*}
where $C>0$ is the embedding constant for $H^1(\R^N)\hookrightarrow L^4(\R^N)$. Thus for any $(u,v)\in\M$, the above estimates yield
\begin{equation}\label{equ:normlowerbound}
 (\|u\|+\|v\|)^2\geq\rho:=\frac{1+\kappa}{2C(\mu_2+\beta)}>0.
\end{equation}
 Therefore
\begin{equation}\label{equ:energylowerbound}
 I_{\kappa,\beta}|_{\M}=\frac{1}{4}\bigg(\|u\|^2+\|v\|^2+2\kappa\intrn uv\bigg)\geq\frac{(1+\kappa)\rho}{8}>0.
\end{equation}}

Step 2: The critical point of $I_{\kappa,\beta}|_{\M}$ is also a critical point of $I_{\kappa,\beta}$.

Assume that $(u_0, v_0)$ is a critical point of $I_{\kappa,\beta}|_{\M}$. Denote by $F(u,v)=\abr{I_{\kappa,\beta}'(u,v),(u,v)}$, then there exists a Lagrange multiplier $\Lambda$ such that
\begin{equation}\label{equ:Lmultiplierequ}
 I'_{\kappa,\beta}(u_0,v_0)=\Lambda F'(u_0,v_0).
\end{equation}
Apply both sides to $(u_0, v_0)$ and note that $(u_0, v_0)\in\M$,
\begin{equation}\label{equ:Lmultiplierzero}
 0=\abr{I'_{\kappa,\beta}(u_0,v_0),(u_0,v_0)}=\Lambda\abr{F'(u_0,v_0),(u_0,v_0)}.
\end{equation}
Recall the lower bound of $\h$-norm on $\M$ \eqref{equ:normlowerbound}, then
\begin{align*}
 \abr{F'(u_0,v_0),(u_0,v_0)}&=-2(\|u_0\|^2+\|v_0\|^2)-4\kappa\intrn u_0v_0\\
 &\leq-2(1+\kappa)(\|u\|^2+\|v\|^2)\\
 &\leq-(1+\kappa)\rho<0.
\end{align*}
Combine with \eqref{equ:Lmultiplierzero} we get $\Lambda=0$. Now \eqref{equ:Lmultiplierequ} gives $I'_{\kappa,\beta}(u_0,v_0)=0$, i.e. $(u_0, v_0)$ is a critical point of $I_{\kappa,\beta}$.

Step 3: $I_{\kappa,\beta}|_{\M}$ satisfies the PS condition.

Since $\kappa\in(-1,0)$, we introduce a new norm on $\h_r$:
\begin{equation*}
 \|(u,v)\|_1=\sqrt{\|u\|^2+\|v\|^2+2\kappa\intrn uv}.
\end{equation*}
It is easy to verify that
\begin{equation}\label{equ:equivnorm}
 (1+\kappa)\|(u,v)\|_\h\leq\|(u,v)\|_1\leq\sqrt{2}\|(u,v)\|_\h,
\end{equation}
i.e., $\|(\cdot,\cdot)\|_\h$ and $\|(\cdot,\cdot)\|_1$ are
equivalent norms on $\h_r$. Now let $\{(u_n, v_n)\}\subset\h_r$ be a
PS sequence of $I_{\kappa,\beta}|_{\M}$, i.e., there exists $c\in\R$
such that
\begin{equation*}
 I_{\kappa,\beta}|_{\M}(u_n,v_n)\rightarrow c,\hspace{1cm} I'_{\kappa,\beta}|_{\M}(u_n,v_n)\rightarrow0\ \hbox{as}\ n\rightarrow\infty.
\end{equation*}
It is easy to see from \eqref{equ:energylowerbound} and
\eqref{equ:equivnorm} that $\|(u_n, v_n)\|_1$ is bounded,  then
there exists a subsequence of $\{(u_n, v_n)\}$, still denoted by
$\{(u_n, v_n)\}$ for simplicity, which weakly converges to $(u, v)$
in $\h_r$ with the topology induced by $\|(\cdot,\cdot)\|_1$. Using
the compact embedding $H^1_r(\R^N)\hookrightarrow L^4(\R^N)$, \iffalse and
\eqref{equ:equivnorm}, we get $(u_n, v_n)\rightarrow(u,v)$ strongly
in $L^4(\R^N)\times L^4(\R^N)$. Thus\fi we get $\intrn u_n^4\rightarrow\intrn
u^4$, $\intrn v_n^4\rightarrow\intrn v^4$, and by H\"older
inequality, 
\begin{align*}
 \bigg|\intrn(u_n^2v_n^2-u^2v^2)\bigg|&\leq\intrn u_n^2|v_n^2-v^2|+\intrn v^2|u_n^2-u^2|\\
 &\leq|u_n|_4^2\left(\intrn(v_n+v)^2(v_n-v)^2\right)^\frac{1}{2}+|v|_4^2\left(\intrn(u_n+u)^2(u_n-u)^2\right)^\frac{1}{2}\\
 &\leq|u_n|_4^2\cdot|v_n+v|_4\cdot|v_n-v|_4+|v|_4^2\cdot|u_n+u|_4\cdot|u_n-u|_4\\
 &\rightarrow0,
\end{align*}
as $n\rightarrow\infty$. Combine the definition of $\M$ and the above
limits, we get $\|(u_n,v_n)\|_1\rightarrow\|(u,v)\|_1$ as
$n\rightarrow\infty$. The norm convergence and the weak convergence together imply
$(u_n,v_n)\rightarrow(u,v)$ strongly in $\h_r$. Therefore
$I_{\kappa,\beta}|_{\M}$ satisfies the PS condition.

Step 4: $\inf_\M I_{\kappa,\beta}$ can be achieved by a positive and radially
symmetric function.

Let $\{(u_n, v_n)\}$ be a minimizing sequence of
$I_{\kappa,\beta}(u, v)|_\M$. Then there exist a positive sequence
$\{t_n\}$ such that $(t_n|u_n|, t_n|v_n|)\in\M$. Since
$-1<\kappa<0, \beta>0$, there holds
 {\allowdisplaybreaks
\begin{align*}
  t_n^2&=\frac{\|u_n\|^2+\|v_n\|^2+2\kappa\int |u_nv_n|}{\int(\mu_1u_n^4 +\mu_2v_n^4+2\beta\int|u_n|^2|v_n|^2}
  \leq\frac{\|u_n\|^2+\|v_n\|^2+2\kappa\int u_nv_n}{\int(\mu_1u_n^4 +\mu_2v_n^4+2\beta\int u_n^2v_n^2}=1.
 \end{align*}
 According to the definition of $I_{\kappa, \beta}|_\M$,
 \begin{align*}
  I_{\kappa,\beta}(t_n|u_n|, t_n|v_n|) &=\frac{t_n^2}{4} \left(\|u_n\|^2+\|v_n\|^2+2\kappa\int_{\R^N}|u_nv_n| dx\right) \\
  &\leq\frac{1}{4} \left(\|u_n\|^2+\|v_n\|^2+2\kappa\int_{\R^N}u_nv_n dx\right)\\
  &=I_{\kappa,\beta}(u_n, v_n).
 \end{align*}}
 Therefore, we can assume $u_n \geq0, v_n\geq0$.

Denote by $u^*_n, v^*_n$ the Schwartz symmetrization of $u_n$ and $v_n$ respectively. Similar to the above arguments, there exists $t_n^*>0$ such that $(t_n^*u_n^*, t^*_nv_n^*)\in\M$. Since
 \begin{equation*}
  \|u_n^*\|^2\leq\|u_n\|^2,\ \ \ \intrn u_n^*v_n^*\geq\intrn u_nv_n,\ \ \ |u_n^*|_4=|u_n|_4,
 \end{equation*}
the same arguments yields $t^*_n\leq1$ and $I_{\kappa,\beta}(t^*_nu^*_n, t^*_nv_n^*)\leq I_{\kappa,\beta}(u_n, v_n)$.
 Thus we can also assume that the minimizing sequence consists of radial functions. By Step 3, $\inf_\M I_{\kappa,\beta}$ has a positive radial minimizer, $(u_{\kappa,\beta}, v_{\kappa,\beta})$, which gives a critical point of $I_{\kappa, \beta}|_\M$. Then by Step 2, this is also a critical point of $\inf I_{\kappa,\beta}$. By the strong maximum principle we get $u_{\kappa,\beta}>0$ and $v_{\kappa,\beta}>0$, i.e. system \eqref{equ:twoparameters} has a positive ground state solution. %This minimizer of $I_{\kappa,\beta}$ corresponds to a positive ground state solution of \eqref{equ:twoparameters}.

The proof is completed.
\end{proof}

\begin{remark}
 In the estimates of $t_n$ and $t_n^*$, the conditions $\kappa\in(-1,0)$ and $\beta>0$ are used, thus the proof of Lemma \ref{thm:existence-asym} seems hard to be used to find solutions of \eqref{equ:twoparameters} for other values of $\kappa$ or $\beta$.
\end{remark}

For the symmetric system \eqref{equ:symmetric}, we can get solutions on a larger region in $\kappa\beta$-plane.

\begin{lemma}\label{thm:existence-symm}
  System \eqref{equ:symmetric} has at least one positive solutions for every $\kappa>-1$ and $\beta>-1$, {and has multiple positive solutions for $-1<\kappa\leq0$ and $\beta\leq-1$.}
\end{lemma}

\begin{proof} If $\kappa>-1$ and $\beta>-1$, then solution branch $\tw^+$
exists, thus system \eqref{equ:symmetric} has at least one positive
solution.

If $-1<\kappa\leq0$ and $\beta\leq-1$, positive solution are found by using bifurcation method with respect to $\tw^+|_\kappa$ in parameter $\beta$.
 Actually, for each $-1<\kappa\leq0$ fixed, there is a sequence of global bifurcations in radial space $\h_r$. Moreover, these global bifurcation branches are
 unbounded in the negative direction of $\beta$, i.e., the projection of each global bifurcation branch cover the interval $\beta\in(-\infty, -1]$.
 The proofs are analogous to \cite{Bartsch-Dancer-Wang:2010}, so the details are omitted. 
\end{proof}

\begin{remark}
 Due to the presence of parameter $\kappa$, there is no synchronized solution branch if $\mu_1\neq\mu_2$. Thus,
 the proof of Lemma \ref{thm:existence-symm} does not apply to system \eqref{equ:twoparameters} in general.
\end{remark}

\bigskip
\noindent\emph{Proof of Theorem \ref{thm:existence}}~~ It follows
directly from Lemma \ref{thm:nonexistence}, \ref{thm:existence-asym}
and \ref{thm:existence-symm}.\hfill $\square$
\bigskip

Picture~\ref{pic:betabifur} shows the synchronized solution branches $\tw^+|_\kappa$ for three different values of $\kappa$. The bifurcation points on $\tw^+|_\kappa$ coincide with the local bifurcation phenomena demonstrated in \cite{Bartsch-Dancer-Wang:2010}. We omit the global bifurcation branches, which are unbounded in the negative direction of $\beta$, to keep the picture clean.
\begin{figure}[!ht]
 \centering
 \includegraphics[width=0.45\textwidth]{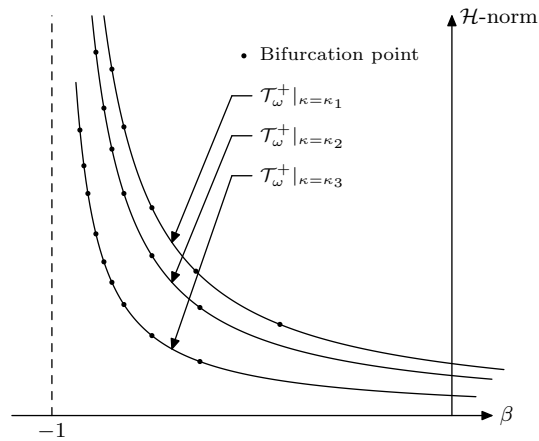}
 \caption{Schematic diagram of $\tw^+|_{\kappa=\kappa_i}$ with $-1<\kappa_3<\kappa_2<\kappa_1=0$.}
 \label{pic:betabifur}
\end{figure}

% Since the bifurcation phenomena with respect to $\tw^+$ and $\tw^-$ are similar, and we are more interested in positive solutions, the following discussion will be given to $\tw^+$.

% \begin{remark}
%  If $\kappa=0$, the synchronized solution branch $\tw^+|_{\kappa=0}$ is the same as the one that has been studied in \cite{Bartsch-Dancer-Wang:2010}, and also similar to the solution branches $\tw$'s discussed in \cite{Tian-Wang:2013-jan, Tian-Wang:2013-feb}. If $\beta=0$, bifurcation phenomena with respect to $\tw^+|_{\beta=0}$ was studied in an unpublished manuscript by E. Abreu and Z.-Q. Wang.
% \end{remark}

\section{Linearized system and possible parameter bifurcation points}
Since the bifurcation phenomena with parameter $\beta$ has been
essentially studied in \cite{Bartsch-Dancer-Wang:2010}, we shall
focus on the bifurcations with parameter $\kappa$ in the rest of
this paper. Also, we shall work in the radially symmetric space
$\h_r$ from now on. For simplicity, let
$\omega_\kappa(x):=\omega(\sqrt{1+\kappa}x)$.

Linearize system \eqref{equ:symmetric} at $(\kappa,\beta,u,v)\in\tw^+$, we get \iffalse system of \eqref{equ:symmetric} along $\tw^+$ takes the following form\fi
   {\allowdisplaybreaks
    \begin{align}
    -\Delta\left(\begin{array}{c}
                  \phi\\
                  \psi
                 \end{array}
           \right)+\left(\begin{array}{cc}
                  1 & \kappa\\
                  \kappa & 1
                 \end{array}
           \right)\left(\begin{array}{c}
                  \phi\\
                  \psi
                 \end{array}
           \right)&=\left(\begin{array}{cc}
                  3u^2+\beta v^2 & 2\beta uv\\
                  2\beta uv & 3v^2+\beta u^2
                 \end{array}
           \right)\left(\begin{array}{c}
                  \phi\\
                  \psi
                 \end{array}
           \right)\label{equ:linearizedsystem}\\
           &=\frac{(1+\kappa)\omega_\kappa^2}{1+\beta}\left(\begin{array}{cc}
                  3+\beta  & 2\beta \\
                  2\beta & 3+\beta
                 \end{array}
           \right)\left(\begin{array}{c}
                  \phi\\
                  \psi
                 \end{array}
           \right),\notag
   \end{align}}
   where $(\phi, \psi)\in \h_r$. %  \subsection{Linearized system reduced}
 Denote the coefficient matrices in \eqref{equ:linearizedsystem} by,
   \begin{equation*}
    D(\kappa)=\left(\begin{array}{cc}
                  1  & \kappa \\
                  \kappa & 1
                 \end{array}
           \right), \hspace{1cm}
    E(\beta)=\left(\begin{array}{cc}
                  3+\beta  & 2\beta \\
                  2\beta & 3+\beta
                 \end{array}
           \right).
   \end{equation*}
It is easy to see that $(1,1)^T$ and $(1, -1)^T$ are eigenvectors of $D(\kappa)$ and $E(\beta)$. Let
\begin{equation*}
           P=\left(\begin{array}{cc}
                  1/\sqrt{2} & 1/\sqrt{2} \\
                  1/\sqrt{2} & -1/\sqrt{2}
                 \end{array}
           \right),\hspace{0.6cm}\hbox{then}\hspace{0.6cm}
           P^{-1}=\left(\begin{array}{cc}
                  1/\sqrt{2} & 1/\sqrt{2} \\
                  1/\sqrt{2} & -1/\sqrt{2}
                 \end{array}
           \right),
   \end{equation*}
   and $D(\kappa), E(\beta)$ can be diagonalized as follows,
   \begin{equation*}
    PD(\kappa)P^{-1}=\left(\begin{array}{cc}
                  1+\kappa  & 0 \\
                  0 & 1-\kappa
                 \end{array}
           \right), \hspace{1cm}
    PE(\beta)P^{-1}=\left(\begin{array}{cc}
                  3(1+\beta)  & 0 \\
                  0 & 3-\beta
                 \end{array}
           \right).
   \end{equation*}
Thus \eqref{equ:linearizedsystem} is equivalent to
{\allowdisplaybreaks
    \begin{align*}
    -\Delta P\left(\begin{array}{c}
                  \phi\\
                  \psi
                 \end{array}
           \right)+\left(\begin{array}{cc}
                  1+\kappa & 0\\
                  0 & 1-\kappa
                 \end{array}
           \right)P\left(\begin{array}{c}
                  \phi\\
                  \psi
                 \end{array}
           \right)&=\frac{(1+\kappa)\omega_\kappa^2}{1+\beta}\left(\begin{array}{cc}
                  3(1+\beta)  & 0 \\
                  0 & 3-\beta
                 \end{array}
           \right)P\left(\begin{array}{c}
                  \phi\\
                  \psi
                 \end{array}
           \right),
   \end{align*}}which can be rewritten as
\begin{equation}\label{equ:linearizeRearranged}
 \left\{\begin{array}{l}
         -\Delta(\phi+\psi)+(1+\kappa)(\phi+\psi)=3(1+\kappa)\omega_\kappa^2(\phi+\psi),\\
         -\Delta(\phi-\psi)+(1-\kappa)(\phi-\psi)=\frac{(1+\kappa)(3-\beta)}{1+\beta}\omega_\kappa^2(\phi-\psi).
        \end{array}
 \right.
\end{equation}
 The non-degeneracy of $\omega$ implies the non-degeneracy of
 $\omega_\kappa$ as a solution of
 $$-\Delta\omega_\kappa +(1+\kappa)\omega_\kappa=(1+\kappa)\omega_k^3,$$
 and the first equation of \eqref{equ:linearizeRearranged} implies that $\phi+\psi=0$. Therefore, the nontrivial solution of \eqref{equ:linearizedsystem} must be in the form $(\phi, -\phi)$. Substitute this possible solution form in the second equation of \eqref{equ:linearizeRearranged}, the linearized system can be reduced to
   \begin{equation*}
    -\Delta\phi+(1-\kappa)\phi=\frac{(1+\kappa)(3-\beta)}{1+\beta}\omega_\kappa^2\phi.
   \end{equation*}
After change variable to $y=\sqrt{1+\kappa}x$ (we still use $\phi$ to denote the unknown function for convenience), the above equation becomes
   \begin{equation}\label{equ:linearizedReduced}
    -\Delta\phi+\frac{1-\kappa}{1+\kappa}\phi=\frac{3-\beta}{1+\beta}\omega^2\phi, \hspace{1cm} \phi\in H^1_r(\R^N).
   \end{equation}
Clearly, the nonzero solutions of \eqref{equ:linearizedReduced}
determine eigenvectors of \eqref{equ:linearizedsystem} and these
eigenvectors take the form $(\phi, -\phi)$.

\subsection{An eigenvalue problem}
To find nontrivial solution of \eqref{equ:linearizedReduced}, we investigate the following
eigenvalue problem,
\begin{equation}\label{equ:eigenvalue-kappa}
 -\Delta\phi+\frac{1-\kappa}{1+\kappa}\phi=\lambda_j(\kappa)\omega^2\phi,\hspace{1cm} \phi\in H^1_r(\R^N).
\end{equation}
Here we denote the eigenvalue by $\lambda_j(\kappa)$ to indicate the dependency on $\kappa$. Let $C(\kappa)=(1-\kappa)/(1+\kappa)$, then $C(\kappa)\rightarrow\infty$ as $\kappa\rightarrow-1^+$ and $C(\kappa)\rightarrow-1$ as $\kappa\rightarrow\infty$. In addition, $C(\kappa)$ is decreasing on the interval $(-1,\infty)$. Denote $\lambda_j:=\lambda_j(0)$. Since $\omega>0$ and $\omega\in L^\infty$, it is well known that
\begin{equation*}
 0<\lambda_1<\lambda_2<...<\lambda_j<...,\ \hbox{and}\ \lambda_j\rightarrow\infty\ \ \hbox{as}\ \ j\rightarrow\infty.
\end{equation*}
% Although we could not find a reference that gives exactly the same conclusion, the next lemma is essentially known. Our proof here is just for completeness.
As we will see in next lemma, $\lambda_j(\kappa)$ is a decreasing
and continuous function of $\kappa$. \iffalse We need to estimate
$\lambda_j(\kappa)$ in the case $\kappa\neq0$ before we can solve
for bifurcation parameters. The estimation is given in the following
lemma.\fi

\begin{lemma}\label{lemma:estimateLambdaKappa}
 For any $j\geq1$, $\lambda_j(\kappa)$ is a continuous and decreasing function of $\kappa$. Moreover, there exists $\kappa_i\in(-1,0]$ such that $\lambda_i(\kappa_i)=\lambda_j$, for every $1\leq i\leq j$.
\end{lemma}

\begin{proof} Recall the variational characterization of $\lambda_j(\kappa)$,
\begin{equation*}
 \lambda_j(\kappa)=\sup_{E_{j-1}}\inf_{\phi\in E_{j-1}^\perp}J(\phi, \kappa),
\end{equation*}
where $E_j$ denotes a $j$-dimensional subspace of $H_r^1(\R^N)$, $E_j^\perp$ denotes the orthogonal space of $E_j$, and
\begin{equation}\label{equ:lambdaCharacterization}
 J(\phi, \kappa):=\frac{\displaystyle \int|\nabla\phi|^2+C(\kappa)\phi^2}{\displaystyle \int\omega^2\phi^2}.
\end{equation}
It is easy to see from \eqref{equ:lambdaCharacterization} that $J(\phi;\cdot)$ is a continuous and decreasing function of $\kappa$, which further implies the continuity and monotonicity of $\lambda_j(\kappa)$. %For completeness, we sketch a proof as follows.

Assume $-1<\kappa_1<\kappa_2$. Consider $\lambda_1(\kappa)$ first. On the one hand,
\begin{align*}
 \lambda_1(\kappa_1)-\lambda_1(\kappa_2)&=\inf_{\phi\in H^1_r(\R^N)} J(\phi, \kappa_1) -\inf_{\phi\in H^1_r(\R^N)} J(\phi, \kappa_2)\\
 &\geq\inf_{\phi\in H^1_r(\R^N)}(J(\phi, \kappa_1)-J(\phi, \kappa_2))\\
 &\geq0.
\end{align*}
On the other hand, let $\phi_2$ be a minimizer of $\inf_{\phi\in
H^1_r(\R^N)} J(\phi, \kappa_2)$, then
\begin{align*}
 \lambda_1(\kappa_1)-\lambda_1(\kappa_2)&=\inf_{\phi\in H^1_r(\R^N)} J(\phi, \kappa_1) -\inf_{\phi\in H^1_r(\R^N)} J(\phi, \kappa_2)\\
 &=\inf_{\phi\in H^1_r(\R^N)} J(\phi, \kappa_1)-J(\phi_2, \kappa_2)
 \leq J(\phi_2, \kappa_1)-J(\phi_2, \kappa_2)\\
 &=(C(\kappa_1)-C(\kappa_2))\frac{\int\phi_2^2}{\int\omega^2\phi_2^2}.
\end{align*}
Hence $\lambda_1$ is a continuous and decreasing function of $\kappa$, by using the continuity and monotonicity of $C(\kappa)$.

Similarly, for $\lambda_j(\kappa)$ with $j\geq2$,
{\allowdisplaybreaks
\begin{align*}
 \lambda_j(\kappa_1)-\lambda_j(\kappa_2)
 &=\sup_{E_{j-1}}\inf_{\phi\in E_{j-1}^\perp}J(\phi, \kappa_1)- \sup_{E_{j-1}}\inf_{\phi\in E_{j-1}^\perp}J(\phi, \kappa_2)\\
 &\leq\sup_{E_{j-1}}\left(\inf_{\phi\in E_{j-1}^\perp}J(\phi, \kappa_1)-\inf_{\phi\in E_{j-1}^\perp}J(\phi, \kappa_2)\right)\\
 &\leq(C(\kappa_1)-C(\kappa_2))\sup_{E_{j-1}} \frac{\int\phi_2^2}{\int\omega^2\phi_2^2}\\
 &\leq\frac{C(\kappa_1)-C(\kappa_2)}{|\omega|_\infty^2},
\end{align*}}
where $\phi_2$ is a minimizer of $J(\phi, \kappa_2)$ in $
E_{j-1}^\perp$ and the last step is due to the fact that
$0<\omega(x)\leq |\omega|_\infty$ for any $x\in\R^N$.

On the other hand, let $E^*_{j-1}$ be the $j-1$ dimensional space corresponding to the eigenfunction of $\lambda_j(\kappa_2)$, then
\begin{align*}
 \lambda_j(\kappa_1)-\lambda_j(\kappa_2)
 &\geq\inf_{\phi\in (E^*_{j-1})^\perp}J(\phi, \kappa_1)-\inf_{\phi\in (E^*_{j-1})^\perp}J(\phi, \kappa_2)\\
 &\geq\inf_{\phi\in (E^*_{j-1})^\perp}(J(\phi, \kappa_1)-J(\phi, \kappa_2))\\
 &\geq0.
\end{align*} Thus $\lambda_j$ is a continuous and decreasing function of $\kappa$ for any $j\geq2$.

By the monotonicity and continuity of $\lambda_j(\kappa)$, to find $\kappa_i\in(-1,0]$ such that $\lambda_i(\kappa_i)=\lambda_j$ for $1\leq i\leq j$, we only need to find a $\kappa\in(-1,0]$ such that $\lambda_1(\kappa)>\lambda_j$ for each $j\geq2$. By the $L^\infty$ boundedness of $\omega$,
 \begin{equation*}
  J(\phi, \kappa)\geq C(\kappa)\frac{\int\phi^2}{\int\omega^2\phi^2}\geq \frac{C(\kappa)}{|\omega|_\infty^2},
 \end{equation*}
which implies that $\lambda_1(\kappa)\geq C(\kappa)/|\omega|_\infty^2$. According to the monotonicity of $C(\kappa)$, we see that $\lambda_1(\kappa)\rightarrow\infty$ as $\kappa\rightarrow-1^+$.

The proof is completed. 
\end{proof}

\begin{remark}\label{remark:about-lambda-kappa}
 For any fixed $j$, Lemma \ref{lemma:estimateLambdaKappa} shows that the eigenvalue $\lambda_j(\kappa)$ can be greater than arbitrary given positive number,
  provided $\kappa$ is close enough to $-1$. On the other hand, as $\kappa\rightarrow\infty$, $\lambda_j(\kappa)$ is decreasing but with a lower bound.
  Therefore, \iffalse $\lambda_j(\kappa)$ can be less than $\lambda_j$ for sufficiently large $\kappa$, \fi it is not guaranteed to have $\lambda_j(\kappa)<\lambda_i$ for any $i<j$, no matter how large $\kappa$ is.
\end{remark}

\subsection{Local bifurcations}
By comparing \eqref{equ:linearizedReduced} and
\eqref{equ:eigenvalue-kappa}, the linearized system
\eqref{equ:linearizedsystem} has a nontrivial solution if $\kappa$
satisfying
\begin{equation}\label{equ:solvebifurParameters}
 f(\beta):=\frac{3-\beta}{1+\beta}=\lambda_j(\kappa)
\end{equation}
for some $j\geq1$. \iffalse In this paper, we first fix $\beta$ and then consider $\kappa$ as the bifurcation parameter.\fi Note that $f$ is decreasing in $\beta$, and
\begin{equation*}
 f(\beta)\rightarrow\infty,\ \hbox{as}\ \beta\rightarrow-1^+;\ \
 f(\beta)\rightarrow-1,\ \hbox{as}\ \beta\rightarrow\infty.
\end{equation*}
The following lemma shows the existence of local bifurcations with respect to $\tw^+|_\beta$.

\begin{lemma}\label{thm:localbifurcation}
 For each fixed $\beta>-1$ such that $f(\beta)\geq\lambda_1$, there are finitely many bifurcation points of \eqref{equ:symmetric} with respect to $\tw^+|_\beta$. {Moreover, denote
 \begin{equation*}
  K(\beta)=\{\kappa_j~|~(\kappa_j, u_{\kappa_j}, v_{\kappa_j})~\hbox{is a bifurcation point of \eqref{equ:symmetric} with respect to}~ \tw^+|_\beta, j\geq1\},
 \end{equation*}   and  $\sharp K(\beta)$ the number of elements of
 $K(\beta)$, then
 $\sharp K(\beta)\rightarrow\infty$ as $\beta\rightarrow-1^+$.}
\end{lemma}

\begin{proof} For fixed $\beta>-1$, we drop the subscript $\beta$ in \eqref{equ:energyfunctional} and write the energy functional as $I_\kappa$.
\iffalse
Recall the energy functional defined in  \eqref{equ:energyfunctional},
 \begin{equation*}
  I_{\kappa}(u,v)=\frac{1}{2}\int_{\R^N}\bigg(|\nabla u|^2+|\nabla v|^2+u^2+v^2+2\kappa uv\bigg)dx-\frac{1}{4}\int_{\R^N}(u^4+v^4)dx- \frac{\beta}{2}\int_{\R^N}u^2v^2,
 \end{equation*}
 where we drop the subscript $\beta$ since $\beta$ is fixed.\fi
 Then the Hessian of $I_\kappa$ is
 \begin{equation*}
  H_\kappa[(\phi,\psi)]^2=\int_{\R^N}\bigg(|\nabla\phi|^2+|\nabla\psi|^2+(\phi^2+\psi^2)+2\kappa\phi\psi\bigg)dx + F(u, v, \phi, \psi, \beta),
 \end{equation*}where $F(u, v, \phi, \psi, \beta)$ denotes the rest terms that do not depend on $\kappa$. The derivative of $H_\kappa$ in $\kappa$, restricted to the kernel space of linearized system \eqref{equ:linearizedsystem} at $(\kappa_j, u_j, v_j)$ is
 \begin{equation*}
  H_\kappa'[(\phi,-\phi)]^2=-2\int_{\R^N}\phi^2dx<0,
 \end{equation*}
where $\phi\in V_j$ and $V_j$ denotes the $j$-th eigenspace of
\eqref{equ:linearizedReduced}. So the Morse index of $I_{\kappa}$ at
$(\kappa, u_\kappa, v_\kappa)$ is strictly increasing in $\kappa$.
In particular, it changes when $\kappa$ passes a value $\kappa_j$
that solves \eqref{equ:solvebifurParameters}. According to
\cite[Theorem 8.9]{Mawhin-Willem:1989}, this $\kappa_j$ is indeed a
parameter bifurcation point.

For fixed $\beta_0\in(-1,\infty)$ such that
$f(\beta_0)\geq\lambda_1$, the monotonicity of $f$ implies that
there exists $i_0$ such that $\lambda_{i_0}\leq
f(\beta_0)<\lambda_{i_0+1}$. Then by Lemma
\ref{lemma:estimateLambdaKappa} and Remark
\ref{remark:about-lambda-kappa}, there exists $\kappa_i\in(-1,0]$
for every $i\leq i_0$ such that $\lambda_i(\kappa_i)=f(\beta_0)$.
Thus there are $i_0$ possible parameter bifurcation points
determined through \eqref{equ:solvebifurParameters}, and
$i_0\rightarrow\infty$ as $\beta_0\rightarrow-1^+$. Therefore,
\begin{center}
 $\sharp K(\beta)<\infty$ for $\beta$ fixed, and $\sharp K(\beta)\rightarrow\infty$ as $\beta\rightarrow-1^+$,
\end{center}
where the monotonicity of $f$ is used.
\end{proof}

\begin{remark}
 In the proof of Lemma \ref{thm:localbifurcation}, we only consider $\kappa\in(-1,0]$. Actually, according to Lemma \ref{lemma:estimateLambdaKappa}, we may also find bifurcation parameters for $\kappa\in(0,\infty)$. More precisely, suppose $f(\beta_0)<\lambda_{j_0}$ for certain $j_0\geq1$ and
 \begin{description}
  \item[(*)] there exist $\kappa_{j_0}>0$ such that $\lambda_{j_0}(\kappa_{j_0})=f(\beta_0)$,
 \end{description}
 then using the same arguments as Lemma \ref{thm:localbifurcation}, $\lambda_{j_0}(\kappa_{j_0})$ is also a parameter bifurcation point. But as it is explained in Remark \ref{remark:about-lambda-kappa}, condition (*) is not always satisfied.
\end{remark}
\begin{remark}
 In radially symmetric space $\h_r$, every eigenvalue of \eqref{equ:eigenvalue-kappa} has multiplicity one.
  Then according to Rabinowitz's global bifurcation theory, these local bifurcations in Lemma \ref{thm:localbifurcation}
  actually give rise to global bifurcations, i.e., there is a continuous solution branch in $\R\times\h_r$ emanating from each bifurcation point on $\tw^+|_\beta$.
\end{remark}
\begin{remark}
 For any fixed $-1<\beta<0$, the synchronized solution branch $\tw^+|_\beta$ approaches $(0,0)$ as $\kappa\rightarrow-1^+$, thus $(-1,-1,0,0)$ is a bifurcation solution of \eqref{equ:symmetric} in product space $\R\times\R\times\h_r$ with two dimensional bifurcation parameter $(\kappa, \beta)\in\R\times\R$.
\end{remark}
\par

Now we give more descriptions for dependence of the bifurcations with respect to $\tw^+|_{\beta}$ on $\beta$.

\begin{lemma}\label{lem:bifurcationEnergy}
 Let $\kappa(\beta)=\min K(\beta)$ and $(\kappa(\beta), u_{\kappa(\beta)}, v_{\kappa(\beta)})\in\tw^+|_\beta$. Then
 \begin{equation*}
  \kappa(\beta)\rightarrow-1^+,\ \ \
  \|u_{\kappa(\beta)}\|_{L^2}=\|v_{\kappa(\beta)}\|_{L^2}\rightarrow\infty,\ \ \
  \hbox{as}\ \beta\rightarrow-1^+.
 \end{equation*}
\end{lemma}

\begin{proof} Assume for contradiction that $\kappa(\beta)\nrightarrow-1$ as
$\beta\rightarrow-1^+$, then there exists $\kappa_0>-1$  such that
$\kappa(\beta)\geq\kappa_0$. Using the monotonicity of
$\lambda_1(\kappa)$ in $\kappa$, we have
$\lambda_1(\kappa(\beta))\leq\lambda_{\kappa_0}$. On the other hand,
$f(\beta)\rightarrow\infty$ as $\beta\rightarrow-1^+$. In
particular, for $f(\beta)>\lambda_{\kappa_0}$ for $\beta$ close to
$-1$. But $f(\beta)=\lambda_1(\kappa(\beta))$, a contradiction. Thus
$$\lim_{\beta\rightarrow-1^+}\kappa(\beta)=-1.$$

Next, we estimate the $L^2$ norm of bifurcation solutions. Recall the definition of $\lambda_1(\kappa)$,
\begin{equation*}
 \lambda_1(\kappa)=\inf_{\phi\in H^1_r(\R^N)}J(\phi;\kappa)=\inf_{\psi\in H^1_r(\R^N)} \frac{C(\kappa)\int(|\nabla\psi|^2+\psi^2)}{\int(\omega_\kappa^*)^2\psi^2},
\end{equation*}
where $\omega_\kappa^*(x)=\omega(x/\sqrt{C(\kappa)})$, $\phi(x)=\psi(x/\sqrt{C(\kappa)})$. From equation \eqref{equ:scalar_omega}, we derive the equation,
\begin{equation*}
-C(\kappa)\Delta\omega_\kappa^*+\omega_\kappa^* =(\omega_\kappa^*)^3.
\end{equation*} Clearly, $\omega_\kappa^*$ is the unique positive radial ground state solution of this equation. Now
{\allowdisplaybreaks
\begin{align*}
 \lambda_1(\kappa)&=\inf_{H^1_r(\R^N)} \left(\frac{\int C(\kappa)|\nabla\psi|^2+\psi^2}{\int(\omega_\kappa^*)^2\psi^2} +\frac{(C(\kappa)-1)\int\psi^2}{\int(\omega_\kappa^*)^2\psi^2}\right)\\
 &\geq\inf_{H^1_r(\R^N)}\frac{\int C(\kappa)|\nabla\psi|^2+\psi^2}{\int(\omega_\kappa^*)^2\psi^2}+\inf_{H^1_r(\R^N)}\frac{(C(\kappa)-1)\int\psi^2}{\int(\omega_\kappa^*)^2\psi^2}\\
 &=\frac{\int C(\kappa)|\nabla\omega_\kappa^*|^2+(\omega_\kappa^*)^2}{\int(\omega_\kappa^*)^4}+\inf_{H^1_r(\R^N)}\frac{(C(\kappa)-1)\int\psi^2}{\int(\omega_\kappa^*)^2\psi^2}\\
 &\geq1+\frac{C(\kappa)-1}{|\omega|_\infty^2}\geq\frac{-2\kappa}{|\omega|_\infty^2(1+\kappa)},
\end{align*}}
where the fact $0<\omega^*_\kappa(x)\leq |\omega|_\infty$ for $x\in\R^N$ is used. \iffalse Thus we have
\begin{equation*}
 \lambda_1(\kappa)\geq 1+\frac{C(\kappa)-1}{|\omega|_\infty^2}\geq\frac{-2\kappa}{|\omega|_\infty^2(1+\kappa)}%=\frac{C(\kappa)+1}{2}=\frac{1}{1+\kappa}.
\end{equation*}\fi Since $f(\beta)=\lambda_1(\kappa(\beta))$,
the above inequality yields $\frac{3-\beta}{1+\beta}=\lambda_1(\kappa(\beta))\geq\frac{-2\kappa(\beta)}{|\omega|_\infty^2(1+\kappa(\beta))}$. Therefore,
\begin{equation}\label{inequ:kappabeta}
 \frac{1+\kappa(\beta)}{1+\beta}\geq\frac{-2\kappa(\beta)}{|\omega|_\infty^2(3-\beta)},
\end{equation}
Let $(\kappa(\beta), u_{\kappa(\beta)}, v_{\kappa(\beta)})\in \tw^+|_\beta$ be the bifurcation point corresponding to $\lambda_1(\kappa(\beta))$, thus
\begin{equation*}
 u_{\kappa(\beta)}(x)=v_{\kappa(\beta)}(x) =\sqrt{\frac{1+\kappa(\beta)}{1+\beta}}\omega(\sqrt{1+\kappa(\beta)}x).
\end{equation*}
{\allowdisplaybreaks
For fixed $x\in\R^N$, $\omega(\sqrt{1+\kappa(\beta)}x)\rightarrow\omega(0) =|\omega|_\infty$ as $\beta\rightarrow-1^+$. Combing with \eqref{inequ:kappabeta}, we have
\begin{align*}
 \lim_{\beta\rightarrow-1^+}u_{\kappa(\beta)}(x)=
 \lim_{\beta\rightarrow-1^+}v_{\kappa(\beta)}(x)&=
 \lim_{\beta\rightarrow-1^+}\sqrt{\frac{1+\kappa(\beta)}{1+\beta}}\omega(\sqrt{1+\kappa(\beta)}x)\\
 &\geq\lim_{\beta\rightarrow-1^+}\sqrt{\frac{-2\kappa(\beta)}{3-\beta}}\frac{\omega(\sqrt{1+\kappa(\beta)}x)}{|\omega|_\infty}\\
 &=\frac{\sqrt{2}}{2},
\end{align*} for any $x\in\R^N$ fixed. Here we use the fact that $\omega$ is decreasing in radial direction and its maximum is achieved at the origin. The pointwise convergence implies
\begin{equation*}
  \|u_{\kappa(\beta)}\|_{L^2}=\|v_{\kappa(\beta)}\|_{L^2}\rightarrow\infty,\ \ \
  \hbox{as}\ \beta\rightarrow-1^+.
 \end{equation*}
 The proof is completed. }
 \end{proof}

 Picture \ref{pic:kappabifur} shows the local bifurcations
with respect to $\tw^+|_{\beta=\beta_i}$ ($i=1,2,3$) and the
dependence of local bifurcations on $\beta$ , where $\kappa$ is the
bifurcation parameter and $\beta_1>\beta_2>\beta_3>-1$.

\begin{figure}[!ht]
 \centering
 \includegraphics[width=0.45\textwidth]{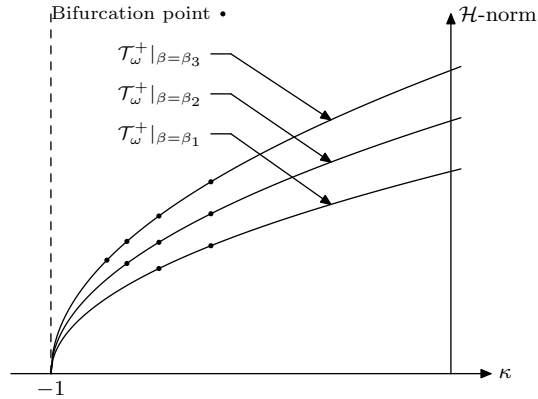}
 \caption{\footnotesize Schematic diagrams of local bifurcations on $\tw^+|_{\beta_j}$ for $\beta_1>\beta_2>\beta_3>-1$.}
 \label{pic:kappabifur}
\end{figure}

\bigskip

Denote by $(\kappa_j(\beta), u_{\kappa_j(\beta)}, v_{\kappa_j(\beta)}) \in\tw^+|_\beta$ the bifurcation point corresponding to
 \begin{equation*}
  f(\beta)=\lambda_j(\kappa_j(\beta)).
 \end{equation*}
By the monotonicity of $\lambda_j(\kappa)$, we have
\begin{equation*}
 -1<\kappa_1(\beta)<\kappa_2(\beta)<\cdots.
\end{equation*}
Based on these inequalities, \eqref{inequ:kappabeta} and the first part of proof of Lemma \ref{lem:bifurcationEnergy}, one can easily get the following corollary.

\begin{corollary}\label{coro:bifurcationEnergy}
 Denote $(\kappa_j(\beta), u_{\kappa_j(\beta)}, v_{\kappa_j(\beta)})\in\tw^+|_\beta$ the bifurcation point corresponding to
 \begin{equation*}
  f(\beta)=\lambda_j(\kappa_j(\beta)).
 \end{equation*}Then for each $j$, we have
 \begin{equation*}
  \kappa_j(\beta)\rightarrow-1,\ \ \|u_{\kappa_j(\beta)}\|_{L^2}= \|v_{\kappa_j(\beta)}\|_{L^2}\rightarrow\infty,
 \end{equation*} as $\beta\rightarrow-1^+$.
\end{corollary}

% \tb{
% \begin{remark}
%  Denote $(\kappa_j(\beta), u_{\kappa_j(\beta)}, v_{\kappa_j(\beta)})\in\tw^+$ the bifurcation point corresponding to
%  \begin{equation*}
%   f(\beta)=\lambda_j(\kappa_j(\beta)).
%  \end{equation*}Then we have
%  \begin{equation*}
%   \kappa_j(\beta)\rightarrow-1,\ \ \|u_{\kappa_j(\beta)}\|_{L^2}= \|v_{\kappa_j(\beta)}\|_{L^2}\rightarrow\infty,
%  \end{equation*} as $\beta\rightarrow-1^+$.
% \end{remark}}

\section{Global bifurcations}
Since each eigenvalue $\lambda_j(\kappa)$ of \eqref{equ:eigenvalue-kappa} has multiplicity one in $\h_r$. By Rabinowitz's global bifurcation theorem \cite{Rabinowitz:1971}, there is a global bifurcation branch emanates from $\tw^+|_\beta$ at each bifurcation point. In this section, we shall discuss these global bifurcations.

In order to show the existence of positive bifurcation branches, we consider a modified system of \eqref{equ:symmetric},
\begin{equation}\label{equ:symmetric-cutoff}
  \left\{\begin{array}{l}
           -\Delta u + u + \kappa v = (u^+)^3 + \beta uv^2,\\
           -\Delta v + v + \kappa u = (v^+)^3 + \beta u^2v,\\
           u, v\in H_r^1(\R^N),
          \end{array}
   \right.
 \end{equation}
where $u^+=\max\{u,0\}, u^-=\min\{u,0\}$.   About the nontrivial
solutions of \eqref{equ:symmetric-cutoff}, we have the following
lemma.

\begin{lemma}\label{lem:positivity}
 For any $\beta\in(-1,0]$ and $\kappa\in(-1,0]$, the nontrivial solutions of \eqref{equ:symmetric-cutoff} are positive, i.e., $u>0, v>0$ in $\R^N$.
\end{lemma}

\begin{proof} Multiply both sides of the first equation of \eqref{equ:symmetric-cutoff} by $u^-$ and both sides of the second equation by $v^-$, then integrate on $\R^N$,
\begin{align*}
 \|u^-\|^2+\int\kappa u^-v^-\leq\|u^-\|^2+\int\kappa u^-v&=\beta\int(u^-)^2v^2\leq0,\\
 \|v^-\|^2+\int\kappa u^-v^-\leq\|v^-\|^2+\int\kappa uv^-&=\beta\int u^2(v^-)^2\leq0.
\end{align*}
Adding the above two inequalities together, and note $-1<\kappa\leq0$ we get
\begin{equation*}
 (1+\kappa)(\|u^-\|^2+\|v^-\|^2)\leq0.
\end{equation*}
Therefore, $u\geq0, v\geq0$ in $\R^N$.

From the first equation of \eqref{equ:symmetric-cutoff} and the non-negativity of $v$, we get
\begin{equation*}
\Delta u - (1-\beta v^2)u = -(u^+)^3 + \kappa v\leq0.
% -\Delta u + u = (u^+)^3 +(\beta uv-\kappa)v\geq0\ \ \ \hbox{in}\ B^c_R.
\end{equation*}
By the strong maximum principle, $u>0$ in $\R^N$. Similarly, $v>0$ in $\R^N$.
\end{proof}

\bigskip

According to Lemma \ref{lem:positivity}, it is easy to see that
nontrivial solutions of \eqref{equ:symmetric-cutoff} are positive
solutions of \eqref{equ:symmetric}. On the other hand, nonnegative
solutions of \eqref{equ:symmetric} are also solutions of
\eqref{equ:symmetric-cutoff}. In particular, the synchronized
solution branch $\tw^+|_\beta$ of
 \eqref{equ:symmetric} is a solution branch of system \eqref{equ:symmetric-cutoff}, and there exists a sequence of bifurcation solutions of \eqref{equ:symmetric-cutoff} along $\tw^+|_\beta$.

\begin{remark}\label{rem:relation_original-cutoff}
Denote by $\ms_l^\beta$ the global bifurcation branch of \eqref{equ:symmetric-cutoff} emanating from $\tw^+|_\beta$ at the $l$-th bifurcation point $(\kappa_l, u_{\kappa_l},
v_{\kappa_l})$. Lemma \ref{lem:positivity} implies that if
$\beta\in(-1,0]$ fixed, then for any $(\kappa, u, v)\in\ms_l^\beta$ with
$-1<\kappa\leq0$, there holds $u>0, v>0$. By the relation between
\eqref{equ:symmetric} and \eqref{equ:symmetric-cutoff}, system
\eqref{equ:symmetric} has a global bifurcation branch with positive
components for fixed $\beta\in(-1, 0]$ and parameter $\kappa\in(-1,
0]$. Whereas, the bifurcation branches with respect to $\tw^+|_\beta$ may continue beyond $\kappa=0$.
\end{remark}

\bigskip

\noindent\emph{Proof of Theorem \ref{thm:bifurcation}}~~ It follows
from Lemma \ref{thm:localbifurcation}, Lemma
\ref{lem:bifurcationEnergy}, Corollary \ref{coro:bifurcationEnergy},
Lemma \ref{lem:positivity} and Remark
\ref{rem:relation_original-cutoff}. \hfill $\blacksquare$

\begin{remark}
It is well known that a global bifurcation branch, in the sense of
Rabinowitz, either is unbounded in the product space or contains
multiple bifurcation points on $\tw^+|_\beta$. In
\cite{Bartsch-Dancer-Wang:2010, Tian-Wang:2013-jan,
Tian-Wang:2013-feb}, there are some descriptions in this
respect. So far, corresponding result for $\kappa$-bifurcation of system \eqref{equ:symmetric} or \eqref{equ:symmetric-cutoff} is still unknown.
\end{remark}

\Acknowledgements{The first author is supported by the China Postdoctoral Science Foundation. The second author is supported in part by the National Natural Science Foundation of China 11325107, 11271353, 11331010.}


\end{document}